\documentclass{amsart}
\usepackage{amssymb,amsmath}
\usepackage{graphicx}
\usepackage{enumerate}
\usepackage{ifthen}

\DeclareMathOperator{\tr}{tr}

\DeclareMathOperator{\dist}{dist}

\DeclareMathOperator{\graph}{graph}

\def\R{{\mathbb{R}}}
\def\N{{\mathbb{N}}}

\def\S{{\mathbb{S}}}

\def\theta{{\vartheta}}
\def\phi{{\varphi}}
\def\epsilon{{\varepsilon}}

\newcommand{\A}[1]{\ifthenelse{#1 = 2}{\lvert A\rvert^{#1}}{\tr A^{#1}}}

\mathchardef\ordinarycolon\mathcode`\:
  \mathcode`\:=\string"8000
  \begingroup \catcode`\:=\active
    \gdef:{\mathrel{\mathop\ordinarycolon}}
  \endgroup

\newtheorem{theorem}{Theorem}[section]
\newtheorem{lemma}[theorem]{Lemma}

\newtheorem{corollary}[theorem]{Corollary}

\theoremstyle{definition}
\newtheorem{definition}[theorem]{Definition}

\newtheorem{remark}[theorem]{Remark}

\numberwithin{equation}{section}
\begin{document}
\title[Entire graphs evolving by powers of the mean curvature]{Existence of convex entire graphs evolving by powers of the mean curvature}

\author{Martin Franzen}
\address{Martin Franzen, Universit\"at Konstanz,
   Universit\"atsstrasse 10, 78464 Konstanz, Germany}
\curraddr{}
\def\ukaddress{@uni-konstanz.de}
\email{Martin.Franzen\ukaddress}
\thanks{We would like thank Oliver Schn\"urer for support.\\ The author is a member of the DFG priority program SPP 1489.}

\subjclass[2000]{53C44}

\date{December 19, 2011.}


\keywords{}

\begin{abstract}
We study convex entire graphs evolving with normal velocity equal to a positive power of the mean curvature. Under mild assumptions we prove longtime existence.	
\end{abstract}

\maketitle

\section{Introduction}

We consider the geometric evolution equation
\begin{align}\label{eq:mean}
\frac{d}{dt} X = -H^\rho\nu,
\end{align}
where $\rho > 0$. For $\rho = 1$ this is the well-known mean curvature flow. We show longtime existence for strictly convex graphical solutions fulfilling the $\nu$-condition.\par

Our main theorem is
\begin{theorem}\label{th:existence}
Let $\rho > 0$. Let $u_0 \in C_{loc}^{2,\beta}\left(\R^n\right)$ be strictly convex for some $0 < \beta < 1$. Assume that for every $\epsilon > 0$ there exists $r > 0$ such that for points $p, q \in \graph u_0$,  $\vert \nu(p) - \nu(q) \vert < \epsilon$ if $\vert p-q \vert < 1$ and $\vert p \vert,\,\vert q \vert \geq r$ ($\nu$-condition). Then there exists a convex, strictly mean convex solution
\begin{align*}
u \in C_{loc}^{2;1}\left(\R^n \times (0,\infty)\right) \cap C_{loc}^0\left(\R^n \times [0,\infty)\right)
\end{align*}
to \eqref{eq:graph} --- the graphical formulation of \eqref{eq:mean}.
\end{theorem}

In \cite{gh:flow}, Huisken proved existence for closed, convex hypersurfaces for $\rho = 1$. These surfaces stay convex and contract to a point in finite time. Schulze \cite{fs:evolution} generalized these results to $\rho > 0$.
In the graphical setting Ecker and Huisken studied existence and asymptotic behavior for $\rho = 1$ \cite{eh:mean}.\par

Our proof closely follows \cite{su:gauss} --- a paper with the working title \emph{Gau\ss\;curvature flows near cones} by Schn\"urer and Urbas. In chapter \ref{sc:local}, however, we study a different test function to show local $C^2$-estimates.\par

The rest of the paper is organized as follows. We first state geometric prelimaries and some notation. Then we list the evolution equations needed in Section \ref{sc:local}. There we apply the maximum principle to get upper velocity bounds which imply local $C^2$-estimates. In Section \ref{sc:coordinate} we explore the $\nu$-condition and a special family of coordinate systems suitable for the local $C^2$-estimates. For the lower bounds we use a Harnack inequality, see Section \ref{sc:lower}. Eventually we prove longtime existence.

\section{Geometric Preliminaries}\label{sc:geometric}

\begin{theorem}[\cite{su:gauss}]\label{th:graph}
  Let $\Omega \in \R^{n+1}$ be an open convex unbounded set, $\emptyset \neq \Omega \neq \R^{n+1}$. Assume that $\partial\Omega \in C_{loc}^2$ and that the principal curvatures of $\partial\Omega$ are everywhere strictly positive. Then there exists a rotation R of $\R^{n+1}$, an open convex set $U \subset \R^n$ and $u \in C_{loc}^2(U)$ such that $\graph u\vert_U = R(\partial\Omega)$. 
\end{theorem}
\begin{proof}
 We may assume that $0 \in \Omega$. For all $k \in \N$ we find points $y_k \in \Omega$ such that $\vert y_k \vert \geq k$. We may assume that $\frac{y_k}{\vert y_k \vert}$ converges to $z \in \S^n$ and, after a rotation, that $z = e_{n+1} = (0,\ldots,0,1)$. Let $\pi$ denote the orthogonal projection of $\R^{n+1}$ onto $\R^n \times \{0\}$. We identify $\R^n \times \{0\}$ and $\R^n$. Define $U := \pi(\Omega)$.\par
 Let $x \in \partial\Omega$ and $\nu(x)$ be the outer unit normal to $\Omega$ at $x$. Then $\nu^{n+1}(x) < 0$ for otherwise the strict convexity of $\partial\Omega$ contradicts $y_k \in \Omega$ for $k$ large enough. Hence every line $\{a\} \times \R$, $ a\in\R^n$, intersects $\partial\Omega$ at most once. None of these lines is completely contained in $\Omega$ as $\Omega \neq \R^{n+1}$ and $\nu^{n+1}(x) < 0$ for any $x \in \partial\Omega$. As $\nu^{n+1}(x) < 0$ for $x \in \partial\Omega$, $x+e_{n+1} \in \Omega$ and hence $\pi(x) \in U$. We conclude that $\partial\Omega$ can be written as $\graph u\vert_U$ with $u \in C_{loc}^2(U)$.
\end{proof}

\section{Notation}\label{sc:notation}

Let $\Omega \subset \R^n$. A function $u: \Omega \to \R$ is said to be convex if its epigraph $\{(x,y) \in \R^n \times \R : y > u(x)\}$ is a convex set. We say that a convex function $u: \Omega \to \R$ is strictly convex, if its Hessian $D^2u = (u_{ij})$ has postive eigenvalues. A function $u: \Omega \times [0,\infty)$ is said to be (strictly) convex, if $u(\cdot,t)$ is (strictly) convex for each $t$.\par
We say that a function $u$ solving a parabolic equation is in in $C^2$, if $u(\cdot,t)$ is in $C^2$ for every $t$. The space $C^{2;1}$ denotes those functions, where in addition all first derivatives are continuous.\par
We use Greek indices running from $1$ to $n+1$ for tensors in $(n+1)$-dimensional Euclidean space. It should not cause any problems that we also use $H^\rho$ to denote the normal velocity and that $\beta$ also appears in section \ref{sc:local}. Latin indices refer to quantities on hyperspaces and run from $1$ to $n$. The Einstein summation convention is used to sum over pairs of upper and lower indices unless we write explicit sums. We raise and lower indices of tensors with the respective metrics or its inverses. A dot indicates  a time derivative, e.g. $\dot{u}$.\par
In Euclidean space we will only use coordinate systems which differ from the standard coordinate system by a rigid motion. Therefore its metric is given by $(\bar{g}_{\alpha\beta}) = \text{diag}(1,\ldots,1)$ and the Codazzi equations imply that the first covariant derivative of the second fundamental form is completely symmetric. We use $X=X(x,t)$ to denote the embedding vector of a manifold $M_t$ into $\R^{n+1}$ and $\frac{d}{dt}X = \dot{X}$ for its total time derivative. It is convenient to identify $M_t$ and its embedding in $\R^{n+1}$. An embedding induces a metric $\left(g_{ij}\right)$.\par
We will consider hypersurfaces $M$ that can be represented as $\graph u$ for some function $u: \R^n \to \R$. Let us use $u_i$ to denote partial derivatives of $u$. Using the Kronecker delta $\delta\,\ddot{}$, we have $u_i\delta^{ij}u_j \equiv u_i u^i = \vert Du \vert^2$. The induced metric $\left(g_{ij}\right)$ of $\graph u$ and its inverse $\left(g^{ij}\right)$ are given by
\begin{align*}
g_{ij} = \delta_{ij} + u_i u_j \qquad \text{and} \qquad g^{ij} = \delta^{ij}  - \frac{u^i u^j}{1 + \vert Du \vert^2},
\end{align*}
respectively.\par
We choose $\left(\nu^\alpha\right)$ to be the downwards directed unit normal vector to $M_t$. If $M_t$ is locally represented as $\graph u$, we get
\begin{eqnarray*}
\left(\nu^\alpha\right) = \frac{(Du,-1)}{\sqrt{1+\vert Du \vert^2}}.
\end{eqnarray*}
The embedding also induces a second fundamental form $\left(h_{ij}\right)$. In the graphical setting it is given in terms of partial derivatives by $h_{ij} = u_{ij}/\sqrt{1+\vert Du \vert^2}$. We denote its inverse by $\tilde{h}^{ij}$.\par
We write Latin indices, sometimes preceded by semicolons, e.g. $h_{ij;k}$, to indicate covariant differentiation with respect to the induced metric. Setting $X^\alpha_{;ij} := X^\alpha_{,ij} - \Gamma^k_{ij}X^\alpha_k$, where a comma indicated partial derivatives, the Gau\ss\;formula is
\begin{eqnarray*}
X^\alpha_{;ij} = -h_{ij}\nu^\alpha
\end{eqnarray*}
and the Weingarten equation is
\begin{eqnarray*}
\nu^\alpha_{;i} = h^k_i X^\alpha_k \equiv h_{il} g^{lk} X^\alpha_k.
\end{eqnarray*}
\par
The eigenvalues of $h_{ij}$ with respect to $g_{ij}$ are the principal curvatures of the hypersurface and are denoted by $\lambda_1,\ldots,\lambda_n$. A hypersurface is called convex, if it is contained in the boundary of a convex body. It is called strictly convex, if it is convex and all principal curvatures are strictly positive.\par
The mean curvature is the sum of the principal curvatures
\begin{eqnarray*}
H = \lambda_1 + \ldots + \lambda_n = h_{ij}g^{ij}.
\end{eqnarray*}
For graphical solutions, the initial value problem for the mean curvature flows \eqref{eq:mean} can be written as follows
\begin{eqnarray}\label{eq:graph}
\begin{cases}
  \dot{u} = \sqrt{1 + \vert Du \vert^2} \left(\text{div}\left(\frac{Du}{\sqrt{1+\vert Du \vert^2}}\right)\right)^\rho  & \text{ in } \R^n \times (0,\infty),\\
 u(\cdot,0) = u_0 & \text{ in } \R^n.
\end{cases}
\end{eqnarray}
It is a parabolic equation if and only if $u$ is strictly convex.\par
Let us also define the Gau\ss\;curvature $K = \frac{\det h_{ij}}{\det g_{ij}} = \lambda_1\cdots\lambda_n$. We define $F^{ij} := \frac{\partial F}{\partial h_{ij}} = \rho H^{\rho-1}g^{ij}$.\par
For tensors $A$ and $B$, $A_{ij} \geq B_{ij}$ means that $\left(A_{ij} - B_{ij}\right)$ is positive definite. Finally, we use $c$ to denote universal, estimated constants.\par
In order to compute evolution equations, we use the Gau\ss\;equation and the Ricci identity for the second fundamental form
\begin{eqnarray*}
\begin{array}{l}
 R_{ijkl} = h_{ik}h_{jl} - h_{il}h_{jk},\\
 h_{ik;lj} = h_{ik;jl} + h^a_k R_{ailj} + h^a_i R_{aklj}.
\end{array}
\end{eqnarray*}

\section{Evolution Equations}\label{sc:evolution}

Recall, e.g. \cite{cg:closed,cg:curvature,gh:flow,os:surfaces}, that for a hypersurface moving according to
\begin{eqnarray*}
\frac{d}{dt} X^\alpha = -H^\rho \nu^\alpha \equiv -F\nu^\alpha,
\end{eqnarray*}
we have
\begin{equation*}
 \frac{d}{dt} X^\alpha-F^{ij} X^\alpha_{;ij} = \left(F^{ij} h_{ij}-F\right)\nu^\alpha,
\end{equation*}
\begin{equation*}
 \frac{d}{dt} F - F^{ij} F_{;ij} =FF^{ij}h^k_i h_{kj}.
\end{equation*}
We define $\eta_\alpha := (0,\ldots,0,1)$, $\tilde{v} := -\eta_\alpha \nu^\alpha$, $v := \tilde{v}^{-1}$ and obtain the following evolution equation
\begin{equation*}
\frac{d}{dt} v - F^{ij}  v_{;ij} = -v F^{ij} h^k_i h_{kj} - 2 \frac{1}{v} F^{ij} v_i v_j.
\end{equation*}

\section{Local $C^2$-Estimates}\label{sc:local}

The idea is to use special coordinate systems to prove local a priori estimates on some sphere $\partial B_r(0)$. Standard techniques then imply a priori estimates in $B_r(0)$.\par
We obtain local a priori estimates similar to \cite{bs:entire,gt:elliptic,ap:the}. For hypersurfaces which can be represented as a graph with small gradient, we get the following local $C^2$-estimates.

\begin{theorem}
Let $\rho > 0$, $T > 0$ and $\left(M_t\right)_{t\in[0,T]}$ be a family of complete strictly convex $C^2$-hypersurfaces solving \eqref{eq:mean}. Pick a coordinate system such that each $M^-_t := M_t \cap \{x^{n+1} \leq 0\}$ can be written as $\graph u(\cdot,t)$ in some domain with $\vert  Du(\cdot,t) \vert \leq G$ in $M^-_t$. If $M^-_0$ is bounded, $\beta = \beta(n,\rho) \geq 1$ is large and $G = G(n,\rho,\beta) > 0$ is small enough, then
 \begin{eqnarray*}
\left(-\eta_\alpha X^\alpha\right)^\rho F e^{\beta v \rho}
\end{eqnarray*}
is bounded in $M^-_t$ by the maximum of $c = c(n,\rho,\beta,G,\max -\eta_\alpha X^\alpha)$ and its value at $t = 0$.
\end{theorem}

We omit the proof as it is similar to the proof of the following theorem that allows also to localize in time.

\begin{theorem}
Let $\rho > 0$, $T > 0$ and $\left(M_t\right)_{t\in[0,T]}$ be a family of complete strictly convex $C^2$-hypersurfaces solving \eqref{eq:mean}. Pick a coordinate system such that each $M^-_t := M_t \cap \{x^{n+1} \leq 0\}$ can be written as $\graph u(\cdot,t)$ in some domain with $\vert Du(\cdot,t) \vert \leq G$ in $M^-_t$. If $M^-_0$ is bounded, $\beta = \beta(n,\rho) > 1$ is large and $G = G(n,\rho,\beta) > 0$ is small enough, then
 \begin{eqnarray*}
t^\rho\left(-\eta_\alpha X^\alpha\right)^\rho F e^{\beta v \rho}
\end{eqnarray*}
is bounded by $c = c(n,\rho,\beta,G,\max -\eta_\alpha X^\alpha,\inf F,T)$.
\end{theorem}

\begin{proof}
Let $\tau > 0$. We want to apply the parabolic maximum principle to
\begin{equation*}
  \tilde{w}:=t (-\eta_\alpha X^\alpha)F^\tau e^{\beta v} \qquad \text{in } \bigcup_{0 \leq t \leq T} M_t^-.
\end{equation*}
Consider a point $(x_0,t_0) \in M_{t_0}^-$ such that $\tilde{w}(x_0,t_0) \geq \tilde{w}(x,t)$ for all $0 \leq t \leq t_0$ and $x \in M_t^-$. We may assume that $t_0>0$ and $\tilde{w}(x_0,t_0)>0$. Choose new coordinates around $x_0 \in M_{t_0}^-$ such that $g_{ij}(x_0,t_0)=\delta_{ij}$ and $h_{ij}(x_0,t_0)$ is diagonal with $h_{11} \geq h_{ii}$ for $i=2,\ldots,n$. As in \cite{cg:closed}, we may consider
\begin{equation*}
  w:=\log t+\log(-\eta_\alpha X^\alpha)+\log F^\tau+\beta v
\end{equation*}
instead of $\tilde{w}$. We obtain at $(x_0,t_0)$
\begin{equation*}
  0\leq\dot{w}=\frac{1}{t}+\frac{-\eta_\alpha\dot{X}^\alpha}{-\eta_\alpha X^\alpha} + \tau \frac{\dot{F}}{F} + \beta \dot{v},
\end{equation*}
\begin{equation}
 \label{eq:maximum}
  0=w_i=\frac{-\eta_\alpha X_i^\alpha}{-\eta_\alpha X^\alpha}+\tau\frac{F_{;i}}{F}+\beta v_i,
\end{equation}
\begin{equation*}
  0\geq w_{;ij}=\frac{-\eta_\alpha X_{;ij}^\alpha}{-\eta_\alpha X^\alpha}+\tau\frac{F_{;ij}}{F}+ \beta v_{;ij}-\frac{\eta_\alpha X_i^\alpha \eta_\gamma X_j^\gamma}{(-\eta_\alpha  X^\alpha)^2}-\tau\frac{F_{;i}F_{;j}}{F^2},
\end{equation*}
\begin{equation*}
  0\leq\dot{w}-F^{ij}w_{;ij}.\\
\end{equation*}
\par
Due to the diagonal form of $F^{ij}(x_0,t_0)$ and according to section \ref{sc:evolution} we get
\begin{align*}
0\leq & \, \frac{1}{t} + \frac{1}{v}\left(F^{ij} h_{ij}-F\right)\frac{1}{-\eta_\alpha X^\alpha} + \tau\frac{1}{F}FF^{ij}h^2_{ij}\\
& \, +  \beta\left(-v F^{ij}h^2_{ij}-2\frac{1}{v}\sum_i F^{ii} v^2_i\right) + \sum_iF^{ii}\left(\frac{-\eta_\alpha X_i^\alpha}{-\eta_\alpha  X^\alpha}\right)^2 \\
& \, +  \frac{1}{\tau}\sum_iF^{ii}\left(\frac{-\eta_{\alpha} X_i^\alpha}{-\eta_\alpha X^\alpha}+ \beta v_i\right)^2 \qquad \text{by } (\ref{eq:maximum}).
\end{align*}
We have strict convexity, thus
\begin{align*}
0\leq & \, \frac{1}{t} +  \frac{1}{v}\left(F^{ij} h_{ij}-F\right)\frac{1}{-\eta_\alpha X^\alpha} + \left(\tau-\beta v\right)F^{ij}h^2_{ij}\\
& \, +  \left(\frac{1}{\tau}\beta^2-2\beta\frac{1}{v}\right)v^4\vert Du\vert^2F^{ij}h_{ij}^2 + \left(1+\frac{1}{\tau}\right)\frac{\vert Du\vert^2}{\left(-\eta_\alpha X^\alpha\right)^2} trF^{ij}\\
& \, +  2\frac{1}{\tau}\beta v^2\frac{\vert Du\vert^2}{-\eta_\alpha X^\alpha} F^{ij}h_{ij},
\end{align*}
where we have assumed $\tau\leq\frac{1}{2}\beta$ and have used that $\vert \eta_\alpha X_i^\alpha\vert =\vert u_i\vert\leq\vert Du\vert\leq v$ and $v_i=-v^2\left(-\eta_\alpha \nu_i^\alpha\right)=v^2\eta_\alpha h_i^kX_k^\alpha=v^2h_{ii}\eta_\alpha X_i^\alpha$. Observe that $v\geq1$. 
We obtain
\begin{align*}
& \, \Lambda F^{ij}h_{ij}^2\equiv\left(-\tau+\beta v\left(1+\left(2-\frac{1}{\tau}\beta v\right) v^2\vert Du\vert^2\right)\right)F^{ij}h_{ij}^2\\
\leq & \, \frac{1}{t}+\frac{1+\frac{2}{\tau}\beta v^2\vert Du\vert^2}{-\eta_\alpha X^\alpha}F^{ij}h_{ij}+\frac{\left(1+\frac{1}{\tau}\right)\vert Du\vert^2}{(-\eta_\alpha X^\alpha)^2}trF^{ij}.
\end{align*}
Set $\bar{w}:=t\left(-\eta_\alpha X^\alpha\right)h_{11}e^{\beta v}$ and let $F=H^\rho$, hence
\begin{align*}
\Lambda \bar{w} \leq & \, \frac{nt\left(-\eta_\alpha X^\alpha\right)e^{\beta v}}{t\rho F} + \frac{n\left(1+\frac{2}{\tau}\beta v^2\vert Du\vert^2\right)t\left(-\eta_\alpha X^\alpha\right)e^{\beta v}}{-\eta_\alpha X^\alpha} \\
& \, + \frac{n^2\left(1+\frac{1}{\tau}\right)\vert Du\vert^2t^2\left(-\eta_\alpha X^\alpha\right)e^{2\beta v}}{t\left(-\eta_\alpha X^\alpha\right)^2h_{11}e^{\beta v}} \\
\leq & \, \frac{c\left(\vert Du \vert,T,\beta,n,\tau\right)}{\tilde{w}} + c\left(\vert Du\vert,T,\beta,\max -\eta_\alpha X^\alpha,\inf F,n,\rho,\tau\right).
\end{align*}
If we fix $\beta$ large enough and then assume $G$ is sufficiently small, so that the $\Lambda$-term becomes bigger than $1$, $\bar{w}$ is bounded. Due to $\frac{1}{n} F^\tau \leq h_{11} \leq n F^\tau$ for $\tau = \frac{1}{\rho}$, $t^{\frac{1}{\tau}}\left(-\eta_\alpha X^\alpha\right)^\frac{1}{\tau} H^\rho e^{\beta v \frac{1}{\tau}}$ is bounded and the claim follows.
\end{proof}

\section{The $\nu$-Condition and Coordinate Systems}\label{sc:coordinate}

\subsection{The $\nu$-condition} We define a class of hypersurfaces for which oscillations of the normal decay at infinity.

\begin{definition}[$\nu$-condition, \cite{su:gauss}]\label{df:nu-condition}
The oscillation of the normal of a $C^1$-hy\-per\-sur\-face $M \subset \R^{n+1}$ is said to decay at infinity, if for every $\epsilon > 0$, there exists $r = r(\epsilon) > 0$ such that for all $p, q \in M \setminus B_r(0)$ such that $\vert p - q \vert < 1$, we have
\begin{align*}
\vert \nu(p) - \nu(q) \vert < \epsilon,
\end{align*}
where $\nu$ denotes a continuous choice of the unit normal vector. We say that such a hypersurface $M$ fulfills the $\nu$-condition.\par
Let $M \subset \R^{n+1}$ be as above. Then $M$ is said to fulfill a uniform $\nu$-condition for $\epsilon > 0$, if for all $p, q \in M$ such that $\vert p-q \vert < 1$, we have $\vert \nu(p) - \nu(q) \vert < \epsilon$.\par
A function $u \in C_{loc}^1(\R^n)$ is said to fulfill a (uniform) $\nu$-condition (for $\epsilon > 0$) if $\graph u$ is a hypersurface that fulfills the (uniform) $\nu$-condition (for $\epsilon > 0$).
\end{definition}
\begin{remark} 
A hypersurface which is close to the cone given by the graph of the function $k: \R^2 \to \R,\, x \mapsto \max\{x^1,-x^1,x^2,-x^2\}$ does not fulfill the $\nu$-condition. The function $u: \R^n \to \R,\, u(x) = \vert x \vert^2$ fulfills the $\nu$-condition.\par
Let $u \in C^1(\R^n)$. Then for each $(x,u(x))$ there exists a rotated coordinate system such that $M = \graph u$ can be written as a graph with small gradient in some neighborhood of $(x,u(x))$. If $u$ fulfills the $\nu$-condition and if $\vert x \vert$ is sufficiently large, this neighborhood contains a ball of large radius. So for each pair of points $p, q \in M$ with $\vert  p-q \vert \geq 1$, there exists a sequence $p = p_1,\ldots,p_k = q,\, p_i \in M$, such that $\vert p_i - p_{i-1} \vert < 1$ and $k \leq 2\vert p-q \vert$. Hence we could replace the condition $\vert p-q \vert < 1$ in Definition \ref{df:nu-condition} by $\vert p-q \vert < R$ for any fixed $R > 0$ without changing the meaning of the $\nu$-condition for entire graphs.
\end{remark}
\begin{lemma}\label{lm:comparison}
Let $T > 0,\,\rho > 0$, and $u \in C_{loc}^{2;1}(\R^n\times(0,T))\cap C_{loc}^0(\R^n\times[0,T])$ be a convex solution to \eqref{eq:graph}. Assume that $u \geq 0$. Let $\epsilon,\,r_\epsilon > 0$. Then there exist $\delta,\,r_\delta$, depending only on $\left(\epsilon,\,r_\epsilon,\,T,\,\rho,\,n\right)$ such that
\begin{align*}
\sup_{B_{r_\epsilon}(0)\times[0,T]} u \leq \epsilon \quad \text{if} \quad \sup_{B_{r_\delta}(0)\times\left\{0\right\}} u \leq \delta.
\end{align*}
\end{lemma}
\begin{proof}
The radius of a sphere evolving by \eqref{eq:mean} is given by
\begin{align*}
r(t) = \left(r(0)^{\rho+1}-(\rho+1)(n-1)^\rho t\right)^{\frac{1}{\rho+1}}.
\end{align*}
Assume without loss of generality that $r_\epsilon > 0$ is so big that there exists a positive solution $h$ to 
\begin{align*}
\left(\left(h+\frac{\epsilon}{2}\right)^{\rho+1}-(\rho+1)(n-1)^\rho T\right)^\frac{1}{\rho+1} = \sqrt{h^2 + r_\epsilon^2}.
\end{align*}
\par
This means that a sphere of radius $h+\frac{\epsilon}{2}$ and center $(0,h+\epsilon)$ at time $t=0$, which evolves according to \eqref{eq:mean} contains $\partial B_{r_\epsilon}(0)\times\{\epsilon\}$ at $t=T$. It acts as a barrier if $\graph u$ lies below it at $t=0$. Hence it suffices to choose $\left(\delta,r_\delta\right)=\left(\frac{\epsilon}{2},h+\frac{\epsilon}{2}\right)$.
\end{proof}

\begin{figure}[htbp]
  \centering
    \includegraphics{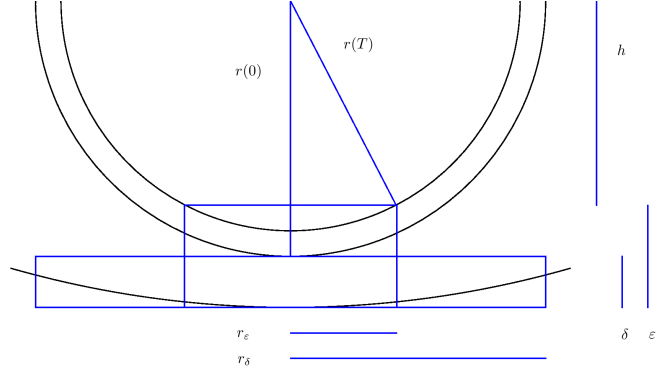}
  \caption{Spheres acting as barriers}
\end{figure}

According to the proof of Lemma \ref{lm:comparison} we immediately get the following generalization which is independent of a graphical representation for all times.
\begin{lemma}\label{lm:comparisonII}
Let $T > 0,\,\rho > 0$. Let $(M_t)_{0\leq t\leq T}$ be a continuous family of complete convex hypersurfaces. Assume that it is smooth for $t > 0$ and fulfills \eqref{eq:mean}. Let $\Omega_t$ denote the closed convex bodies such that $M_t = \partial\Omega_t$ for all $0 \leq t < T$. Let $\epsilon,\,r_\epsilon > 0$. Then there exist $(\delta,r_\delta,h)$, depending only on $(\epsilon,\,r_\epsilon,\,T,\rho,\,n)$ such that $M_t \cap\left(B_{r_\epsilon}(0)\times[0,\epsilon]\right)$ can be written as $\graph u|_{B_{r_\epsilon}(0)}$ with $0\leq u\leq\epsilon$ in $B_{r_\epsilon}(0)$ if
\begin{align*}
B_{h+\frac{\epsilon}{2}} ((0,h+\epsilon)) \subset \Omega_0 \subset \left\{ x^{n+1} \geq 0 \right\}.
\end{align*}
\end{lemma}
\par
We will control the representation as a graph for hypersurfaces that fulfill the $\nu$-condition.
\begin{lemma}[\cite{su:gauss}]\label{lm:comparisonIII}
Let $M \subset \R^{n+1}$ be a complete convex hypersurface of class $C^1$ which fulfills the $\nu$-condition. Let $r \geq 2 > 1 \geq \epsilon > 0$. Then there exists $R > 0$ sufficiently large, depending only on $r,\,\epsilon$ and the $\nu$-condition such that for every $x \in M$ with $\vert x \vert > R$ after a suitable rotation a neighborhood of the origin of $M-x = \left\{y \in \R^n : y + x \in M\right\}$ can be represented as $\graph u|_{B_r(0)}$ with $0 \leq u \leq \epsilon$ in $B_r(0)$.
\end{lemma}
\begin{proof}
We assume without loss of generality that $M-x$ is contained in $\left\{x^{n+1} \geq 0\right\}$. Locally, $M-x$ can be represented as a graph. Due to the convexity of $M$, we find a maximal open set such that there exists a function $u$, which is convex and of class $C^1$, such that $\graph u \subset M-x$. Let $\Omega$ be the subset, where $\vert Du \vert < \frac{\epsilon}{r} \leq \frac{1}{2}$. Observe that $Du(0) = 0$ and $0 \in \Omega$.\par
We claim that for $\vert x \vert$ sufficiently large, $B_r(0) \subset \Omega$. Otherwise consider $x_0 \in \partial\Omega$ with minimal $\vert x_0 \vert$. Pick points $0 = p_0,\,p_1,\ldots,p_k = x_0,\,\vert p_i \vert < \vert x_0 \vert,\, \vert p_i - p_{i+1} \vert < \frac{1}{2}$ and $k \leq 2r+1$. Notice that $\vert (p_i,u(p_i))-(p_{i-1},u(p_{i-1}))\vert < 1$. We may assume that $\vert \nu(p) - \nu(q) \vert < \delta$ for any fixed $\delta > 0$ and $p,\,q \in \graph u|_{B_{\vert x_0 \vert}(0)}$ with $\vert p-q \vert < 1$ by choosing $R = R(\delta,r) > 0$ sufficiently large. Induction yields
\begin{align*}
k\delta \geq \vert \nu(0) - \nu((x_0,u(x_0)))\vert = \left| (0,-1) - \frac{\left(Du(x_0),-1\right)}{\sqrt{1+\frac{\epsilon^2}{r^2}}} \right| \geq \frac{\frac{\epsilon}{r}}{\sqrt{1+\frac{\epsilon^2}{r^2}}} \geq \frac{1}{2} \frac{\epsilon}{r}.
\end{align*}
We obtain $(2r+1)\delta \geq \frac{1}{2} \frac{\epsilon}{r}$ which is impossible for $\delta > 0$ sufficiently small. Hence $\vert Du \vert < \frac{\epsilon}{r}$ in $B_r(0)$.\par
The Lemma follows by integration.
\end{proof}
The following Corollary shows that the $\nu$-condition is preserved during the flow.
\begin{corollary}\label{co:comparison}
Let $T > 0,\,\rho > 0$. Let $\left(M_t\right)_{0\leq t\leq T}$ be a continuous family of complete convex hypersurfaces which is smooth for $t > 0$ and fulfills \eqref{eq:mean} there. Assume that $M_0$ fulfills the $\nu$-condition. Then $M_t$ fulfills the $\nu$-condition for every $0\leq t\leq T$. 
\end{corollary}
\begin{proof}
We want to apply Lemma \ref{lm:comparisonII} to a translated and rotated coordinate system for $r_\epsilon > 0$ big and $\epsilon > 0$ small. According to Lemma \ref{lm:comparisonIII}, we can fulfill the assumption of Lemma \ref{lm:comparisonII} if the origin corresponds to a point outside $B_R$ for $R > 0$ sufficiently large before the change of coordinates. Hence we can apply Lemma \ref{lm:comparisonII}.\par
Now convexity ensures that $\vert Du(x) \vert \leq 2\frac{\epsilon}{r_\epsilon}$ for $\vert x \vert \leq \frac{r_\epsilon}{2}$. Hence for all $0 \leq t \leq T$  the normals of $\graph u|_{B_{r_{\epsilon/2}(0)}|}$ are close to $-e_{n+1}$ and the claim follows.
\end{proof}
For initial data with bounded gradient that fulfill the $\nu$-condition, solutions are unique.
\begin{lemma}
Let $\rho > 0$ and $0 < T \leq \infty$. Let $u_0 \in C_{loc}^1\left(\R^n\right)$ be strictly convex. Assume that $u_0$ fulfills the $\nu$-condition and $\sup_{\R^n} \vert Du_0 \vert < \infty$. Let
\begin{align*}
u_1,\, u_2 \in C_{loc}^{2;1}\left(\R^n \times (0,T)\right) \cap C_{loc}^0\left(\R^n \times [0,T)\right)
\end{align*}
be two strictly convex solutions to \eqref{eq:graph}. Then $u_1 = u_2$.
\end{lemma}
\begin{proof}
By decreasing $T$ if necessary, we may assume without loss of generality that $T < \infty$ and $u_1,\,u_2 \in C^{2;1}\left(\R^n \times (0,T]\right)$. According to Lemma \ref{co:comparison}, both solutions fulfill the $\nu$-condition for any $t \in [0,T]$. The proof of that Lemma implies also that $\sup_{i\in\{1,2\}}\sup_{x\in\R^n}\sup_{t\in [0,T]} \vert Du_i (x,t) \vert < \infty$.\par
Assume that $\sup_{\R^n\times (0,T]} u_1 - u_2 = \epsilon > 0$. According to Lemma \ref{lm:comparisonII} and the uniform gradient bound, the supremum is attained. Hence there exists $(x_0,t_0) \in \R^n \times (0,T]$ such that $(u_1 - u_2)(x_0,t_0) = \epsilon$ and $\sup_{\R^n \times \{t\}} u_1 - u_2 < \epsilon$ for any $t < t_0$. This, however, contradicts the maximum principle for compact domains.
\end{proof}
The following lemma ensures im particular that for computing the distance to embedded tangent planes it suffices to minimize the distance to the embedded tangent plane over a compact set if $u(x) \rightarrow \infty$ for $\vert x \vert \rightarrow \infty$.
\begin{lemma}[\cite{su:gauss}]\label{lm:distance}
Let $u \in C_{loc}^2\left(\R^n\right)$ be strictly convex, $M = \graph u$. Assume that $u(x) \rightarrow \infty$ for $\vert x \vert \rightarrow \infty$. Let $R > 0$ and $x_0 \in B_R(0)$. Denote by $T_xM$ the embedded tangent plane $T_{(x,u(x))}M$. Define the compact set $K$ by
\begin{align*}
K := \partial B_R(0) \cup \left\{x \in \R^n \setminus B_R(0) : u(x) \leq u(x_0)\right\}.
\end{align*}
Then
\begin{align*}
\inf_{\R^n \setminus B_R(0)} \dist\left(T_xM,(x_0,u(x_0))\right) = \inf_K \dist(T_xM,(x_0,u(x_0))).
\end{align*}
Note that this expression is positive and continuous in $x_0$.
\end{lemma}
The assumption $u(x) \rightarrow \infty$ for $\vert x \vert \rightarrow$ is not necessary. We have included it as it makes it easier to write down the definition for $K$.
\begin{proof}
We will first reduce this question to a one-dimensional question. Let $x \in R^n \setminus \left(B_R(0) \cup K\right)$. As the tangent planes $T_{x_0}M$ and $T_xM$ are not parallel, they intersect in an $n-1$-dimensional plane. We may assume that this is invariant under translations by the vectors $e_2,\ldots,e_n$. Hence $Du(x)$ is proportional to $e_1$. Observe that all points, where $Du$ is proportional to $e_1$, lie on a $C^1$-curve $\gamma$. We project the situation orthogonally to the $\left(x^1,x^{n+1}\right)$-plane. Then the boundary of the projection of $M$ is given by the projection of $\graph u|_\gamma$. It suffices to show that $\dist\left(T_xM,(x_0,u(x_0))\right)$ decreases if we move $x$ along $\gamma$ towards $x_0$. Hence if suffices to consider a two dimensional situation as shown in Figure 2. We do not introduce new notations for the projected objects.

\begin{figure}[htbp]
  \centering
    \includegraphics{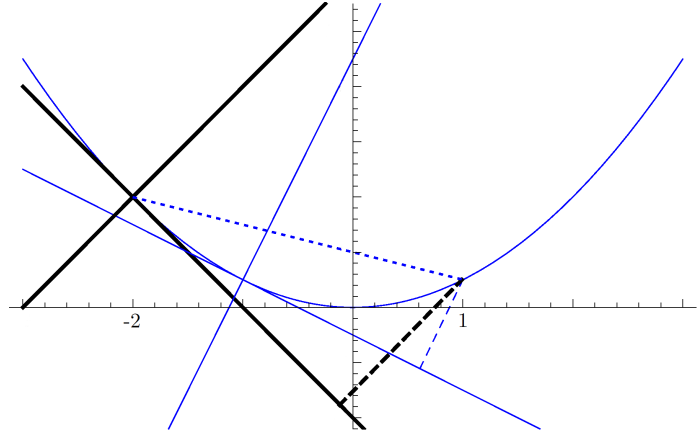}
  \caption{Distance to tangent planes}
\end{figure}

In Figure 2, $x_0 = 1$ and $x = -2$. It shows the tangent plane and normal direction to $\graph u$ at $x$ using a thick line and at the shifted point using a thin line. Dashed lines indicate the distance to the respective tangent lines. The geometric situation, especially $u(x) > u(x_0)$, ensures that the dotted line that connects $(x_0,u(x_0))$ and $(x,u(x))$ and the part of $\graph u$ that connects these two points lie in the same quadrant of the "thick coordinate system" at $x_0$. Hence shifting $x$ towards $x_0$ decreases the distance considered as long as $x$ is outside $K$.\par
Positivity and continuity of the infima considered are clear. 
\end{proof}
\begin{corollary}\label{co:normal}
Let $0 < T_{ex} \leq \infty$, $\rho > 0$. Let
\begin{align*}
u \in C_{loc}^{2;1}\left(\R^n\times [0,T_{ex})\right)\cap C_{loc}^0\left(\R^n\times [0,T_{ex})\right)
\end{align*}
be a strictly convex solution to \eqref{eq:graph}. Assume that $u(\cdot,0)$ fulfills the $\nu$-condition. Let $R > r > 0$. Then there exists $T = T(M_0,r,R) > 0$ with $T\rightarrow\infty$ as $R\rightarrow\infty$ such that $Du\left(B_r(0),t\right)\cap Du\left(R^{n+1} \setminus \overline{B_R(0)},0\right) = \emptyset$ for all $0 \leq t \leq \min\left\{T,T_{ex}\right\}$. This is equivalent to disjointness of the respective normal images.
\end{corollary}
\begin{proof}
Let $M_0 := \graph u(\cdot,0)$. Fix $\vert x \vert > R$, $\vert y \vert < r$. The distance between the embedded tangent plane $T_xM_0$ and $(y,u(y))$ is positive. The considerations above, Lemma \ref{lm:comparison}-Corollary \ref{co:comparison}, ensure that $u(z,t)-u(z,0)$ is bounded for every $t \in [0,T]$ in terms of especially $R$ and $T$ if $z$ is close to $x$. If that distance corresponds to a distance which is smaller than the distance between the embedded tangent plane $T_xM_0$ and $(y,u(y))$, due to convexity we get $\nu(y,t) \neq \nu(x,0)$ and hence $Du(y,t) \neq Du(x,0)$. This establishes a lower bound on $t$ such that $Du(y,t) = Du(x,0)$. The Corollary follows as the estimates are uniform in $r$ and $R$.
\end{proof}

\subsection{Choosing Appropriate Coordinate Systems.} Let us describe a family of coordinate systems that is suitable for the interior $C^2$-estimates. Although a solution $M_t$ to $\frac{d}{dt}X = -H^\rho\nu$ is not known to exist at this stage, some of the following conditions involve the prospective solution $M_t$. This is then to be understood in the sense of a priori estimates, i.e. it has to be ensured that any graphical solution to $\frac{d}{dt} = -H^\rho\nu$ which exists on the time interval $[0,T]$ fulfills these properties.
\begin{definition}[\cite{su:gauss}]\label{df:systems}
Let $R > 0$, $T > 0$, $ G > 0$, $H > 0$. Let $u_0 \in C_{loc}^2\left(\R^n\right)$ and $M_t$ be a strictly convex solution to $\eqref{eq:mean}$. Then a family of coordinate systems $cs_p$ and corresponding coordinates $\left(\tilde{x}^1,\ldots,\tilde{x}^n\right)_p$, $p \in \S^{n-1}$, is a family of $(R,T,G,H)$-coordinate systems for $\left(M_t\right)_{0 \leq t \leq T}$ if the following conditions are fulfilled.
\begin{enumerate}[(i)]
\item $cs_p$ and the original coordinate system of $\R^{n+1}$ differ by a rotation about some point in $\R^{n+1}$.
\item We denote the orthogonal projection of $M_t \cap \left\{\tilde{x}^n<0\right\}$ to $\left\{\tilde{x}^n = 0\right\}$ by $\Omega_{p,t}$ and require that $\bigcup_{0\leq t\leq T} \Omega_{p,t} =: \Omega_p$ is bounded and $\Omega_{p,t} \cap \left(\left\{\lambda p : \lambda > 0\right\} \times \R\right) \neq \emptyset$, where the expressions on the right-hand side refer to quantities in the original coordinate system.
\item 
\begin{align*}
\bigcup_{p,t} M_t \cap \left\{\tilde{x}_p^n < 0\right\} \subset \left(\R^n \setminus B_R(0)\right) \times \R,
\end{align*}
where the expressions on the right-hand side refer once more to quantities in the original coordinate system.
\item In the coordinate system $cs_p$, for each $0 \leq t \leq T$, the set $M_t \cap \left\{\tilde{x}^n < 0\right\}$ can be written as $\graph \tilde{u}(\cdot,t)|_{\Omega_{p,t}}$, where $\tilde{u} \in C^{2;1}\left(\Omega_p \times [0,T]\right)$ is strictly convex.
\item In $\Omega_{p,t}$, we have the gradient estimate $\left\| D\tilde{u} \right\|_{L^\infty} \leq G$.
\item Define $U_p := \graph \tilde{u}|_{\left\{u < -H \right\}}$. Then there exists a constant $r \geq R$ such that
\begin{align*}
\graph u|_{\partial B_r(0)} \subset \bigcup_{p \in \S^{n-1}} U_p.
\end{align*}
\end{enumerate}
\end{definition}
The next Lemma describes a situation where a hypersurface can be represented as a graph with small gradient.
\begin{lemma}[\cite{su:gauss}]\label{lm:boundedness}
Let $M \subset \R^{n+1}$ be a complete convex hypersurface of class $C^2$. Let $r > 0$. Assume for $p,q \in M$ with $\vert p-q \vert < r$ that $\vert \nu(p) - \nu(q) \vert < \epsilon$ for some $\epsilon > 0$ to be chosen. Fix a coordinate system and $\Omega$ maximal such that $M$ is locally given as $\graph u|_\Omega$, $u \in C_{loc}^2(\Omega)$, where $u$ is convex, $0 > u(0) \geq -h$ for some $h > 0$ and $Du(0) = 0$. Let $0 < g \leq 1$. If $\epsilon = \epsilon(h,g,r) > 0$ is chosen sufficiently small, then
\begin{align*}
\vert Du \vert < g
\end{align*}
in $\Omega \cap \{u < 0\}$.
\end{lemma}
It follows from the proof that it suffices that the condition on $\nu$ is fulfilled in the set $\graph u|_{\left\{u\leq 0\right\}}$. Observe that $\{u \leq 0\}$ can be large.
\begin{proof}
Assume that the conclusion was false. Fix $h_g < 0$ minimal such that $\vert Du \vert = g$ somewhere on $\partial \Omega_g \equiv \partial \left\{x \in \Omega : u < h_g\right\}$. Choose $x_0 \in \partial\Omega_g$ with $\vert Du \vert(x_0) = g$. Observe that $\Omega_g$ is convex.\par
 We claim that there exists $\tilde{r} = \tilde{r}(r) > 0$ and $\tilde{\epsilon} = \tilde{\epsilon}(\epsilon)$ with $\tilde{\epsilon}(\epsilon) \rightarrow 0$ as $\epsilon \rightarrow 0$ such that $\vert Du(x) - Du(y) \vert < \tilde{\epsilon}$ for any $x, y \in \Omega_g$ such that $\vert x-y \vert < \tilde{r}$. Let $p = (x,u(x))$ and $q = (y,u(y))$. We obtain
\begin{align*}
\vert p-q \vert^2 &\, = \vert x-y \vert^2 + \vert u(x) - u(y) \vert^2 \\
&\, = \vert x-y \vert^2 + \left| \int_0^1 \left\langle Du(\tau x + (1 - \tau)y),x-y \right\rangle d\tau \right|^2\\
&\, \leq \vert x-y \vert^2 + g^2\vert x - y \vert^2 \leq 2 \vert x-y \vert^2.
\end{align*}
As long as $\nu$ is uniformly strictly contained in the lower hemisphere of $\S^n$ or, equivalently, $\vert Du \vert^2$ is uniformly bounded, the map $\nu = \frac{(-Du,1)}{\sqrt{1+\vert Du \vert^2}} \mapsto Du$ is continuous.
Combining this with the assumption on the normal, the claim follows. In order to simplify the notation, we will assume in the following that $\vert Du(x) - Du(y) \vert < \epsilon$ for any $x,y \in \Omega_g$ with $\vert x-y \vert < r$.\par
Consider the solution $\gamma: [0,\infty) \to \Omega_g$ of the initial value problem
\begin{align*}
\begin{cases}
  \dot{\gamma} = -Du(\gamma(t)), \quad t \geq 0,\\
  \gamma(0) = x_0.
\end{cases}
\end{align*}
It is clear that $\gamma$ exists for all $t$ and $\gamma(t) \rightarrow 0$ for $t \rightarrow \infty$.\par
We will see that $\gamma$ stays for a long interval $[0,t_0]$ in the set $\{\vert Du \vert \geq g/2\}$ which contradicts the boundedness of $u$ from below: Fix $t_0 > 0$ minimal such that $\vert Du \vert(\gamma(t_0)) = g/2$. We obtain
\begin{align*}
u(\gamma(t_0)) - u(\gamma(0)) = \int_0^{t_0} \left\langle Du(\gamma(\tau)),\dot{\gamma}(\tau)\right\rangle d\tau = -\int_0^{t_0} \vert Du \vert^2(\gamma(\tau)) d\tau
\end{align*}
and hence
\begin{align}\label{eq:boundedness}
h \geq \vert u(\gamma(t_0)) - u(\gamma(0)) \vert = \int_0^{t_0} \vert Du \vert^2(\gamma(\tau)) d\tau \geq t_0 \frac{g^2}{4}.
\end{align}
On the other hand, $\gamma(t) \in \Omega_g$ for $t > 0$ and $\vert Du \vert \leq g$ in $\Omega_g$ by definition. This implies that $\vert \dot{\gamma}(t) \vert \leq g$. We get $\vert \gamma(t_1)-\gamma(t_2) \vert < r$ for $0 < t_1, t_2$ with $\vert t_1 - t_2 \vert < r/g$. Our assumption on the gradient implies that
\begin{align*}
\vert Du(\gamma(t_1)) - Du(\gamma(t_2)) \vert < \epsilon \quad \text{for } \vert t_1 - t_2 \vert < \frac{r}{g}.
\end{align*} 
By induction we deduce that
\begin{align*}
\vert Du(\gamma(0)) - Du(\gamma(t))\vert < (k+1)\epsilon \quad \text{for } 0 \leq t \leq k\frac{r}{g}
\end{align*}
for any $k > 0$. In the following, we will assume that $0 < \epsilon < g/4$. We fix $k := \frac{g}{2\epsilon} - 1$ and obtain
\begin{align*}
\vert Du(\gamma(0)) - Du(\gamma(t)) \vert < \frac{g}{2} \quad \text{for } 0 \leq t \leq \frac{t}{4\epsilon}.
\end{align*}
Hence $t_0 > \frac{r}{4\epsilon}$. Now \eqref{eq:boundedness} implies that $h > \frac{rg^2}{16\epsilon}$. This is impossible for $0 < \epsilon \leq \frac{rg^2}{16h}$.
\end{proof}
Now we can show the existence of coordinate systems fulfilling Definition \ref{df:systems}.
\begin{lemma}\label{lm:coordinate}
Let $u \in C_{loc}^2(\R^n)$ be strictly convex. Assume that $u$ fulfills the $\nu$-condition. Let $R > 0$, $T > 0$ and $G > 0$. Then there exists $H > 0$ and a family of $(R,T,G,H)$-coordinate systems for a continuous family $\left(M_t\right)_{0\leq t < \infty}$ of convex hypersurfaces solving \eqref{eq:mean} with $\graph u = M_0$, which is smooth for $t>0$.
\end{lemma}
\begin{proof}
Note that it suffices to fix $\left\{\tilde{x}^{n+1} = 0\right\}$ and $\tilde{e}_{n+1}$ for the new coordinate system.\par
Assume that $R > 0$ is so large that Lemma \ref{lm:boundedness} is applicable in $\R^n \setminus B_R(0)$. According to Lemma \ref{lm:distance}, the distance between any embedded tangent plane to $M_0 = \graph u$ at $(x,u(x))$ with $\vert x \vert > 2 R$ and $(y,u(y))$ with $\vert y \vert \leq R$ is bounded below by $2H$ for some $H > 0$. We claim that we can choose coordinate systems such that $\left\{\tilde{x}^{n+1} = 0\right\}$ is given by the tangent plane to $M_0$ at $(x,u(x))$ with $\vert x \vert > 2R$, shifted by $-2H$ in the direction of $\nu((x,u(x)))$ so that $\tilde{e}_{n+1} = -\nu((x,u(x)))$. Then (i) is clear, (ii) follows from the strict convexity of $u$ at $t = 0$ as $M_t$ is convex and moves only upwards. Condition (iii) is fulfilled by our choice of $H$. Lemma \ref{lm:boundedness} ensures that conditions (iv) and (v) are fulfilled. The $\nu$-condition and Lemma \ref{lm:comparisonII} yield condition (vi).
\end{proof}

\section{Lower Velocity Bounds}\label{sc:lower}

We want to check, that we can apply a Harnack inequality. For $F = F(\lambda_i)$, $\lambda_i > 0$, we define
\begin{align*}
\Phi\left(\kappa_i\right) := -F\left(\kappa_i^{-1}\right).
\end{align*}
We say that $\Phi$ is $\alpha$-concave, if $\Phi = $ sign$\,\alpha\,B^\alpha$ for some $B$, where $B$ is positive and concave. The function $\Phi$ is called the dual function to $F$. 
\begin{lemma}
Let $\rho > 0$. The dual function to $F = H^\rho = \left(\lambda_1 + \ldots + \lambda_n\right)^\rho$ is $-\rho$-concave.
\end{lemma}
\begin{proof}
We obtain the non-zero terms
\begin{align*}
\Phi &\,= -\Delta^\rho,\\
\Phi_i &\, = \frac{\partial \Phi}{\partial \lambda_i} =
\rho\Delta^{\rho-1}\lambda_i^{-2},\\
\Phi_{i,j} &\, = \frac{\partial^2\Phi}{\partial \lambda_i \partial \lambda_j} = -\rho\Delta^{\rho-2}\left(\left(\rho-1\right)\lambda_i^{-2}\lambda_j^{-2} + 2\Delta\lambda_i^{-3}\delta_{ij}\right),
\end{align*}
where $\Delta \equiv \frac{1}{\lambda_1} + \ldots + \frac{1}{\lambda_n}$.\par
According to \cite{ba:harnack}, we have to prove that
\begin{align*}
0 \leq 2\rho\Delta^{\rho-2}\left(-\lambda_i^{-2}\lambda_j^{-2} + \Delta\lambda_i^{-3}\delta_{ij}\right).
\end{align*}
Its principal minors are
\begin{align*}
\frac{\left(2\rho\Delta^{\rho-1}\right)^{k}}{\Delta\left(\Pi_{i=1}^k\frac{1}{\lambda_i}\right)^3}\left(\Delta - \sum_{i=1}^k \frac{1}{\lambda_i} \right) > 0\qquad\text{for } k \leq n
\end{align*}
and the positive definity follows.
\end{proof}
We can apply the Harnack inequality derived in \cite[Theorem 5.17]{ba:harnack}. For our geometric evolution equation \eqref{eq:mean} it reads
\begin{theorem}
Assume that $0 < t_1 < t_2$. Then for smooth compact strictly convex hypersurfaces moving according to \eqref{eq:mean} we have for all $p \in \S^n$ that
\begin{eqnarray}\label{eq:harnack}
\frac{F\left(\nu^{-1}(p),t_2\right)}{F\left(\nu^{-1}(p),t_1\right)} \geq \left(\frac{t_1}{t_2}\right)^{\frac{\rho}{\rho+1}}.
\end{eqnarray}
\end{theorem}
\begin{lemma}
Let $T > 0$, $\rho > 0$, $0 < \beta < 1$. Let $u_0 \in C_{loc}^{2,\beta}(\R^n)$ be strictly convex. Let
\begin{align}
u \in C_{loc}^{2;1}(\R^n \times (0,T)) \cap C_{loc}^0(\R^n \times [0,T])
\end{align}
be a strictly convex solution to \eqref{eq:graph}, obtained as a locally uniform limit of closed hypersurfaces which contain $\graph u_0|_{B_{r_i}(0)}$ initially and lie above $\graph u_0$. Then for every $x \in \R^n$ and every $t_x \in [0,T]$, there exists a positive lower bound $H_x > 0$ which is continuous in $x$ and depends only on $t_x$ and $u_0$ in some neighborhood of $x$, such that for the point $y$ with $\nu((y,u(y,t_x))) = \nu((x,u_0(x)))$, we have $H(y,t_x) \geq H_x > 0$.
\end{lemma}
\begin{proof}
Without changing the notation, let us assume that the approximating hypersurfaces are locally represented as $\graph u$. It is easy to see that the estimate survives the limiting process. Let us assume for simplicity that $t_x = T$.\par
Fix $x_0 \in \R^n$. By a rigid motion, we may assume without loss of generality that $Du_0(x_0) = 0$ and $x_0 = 0$ and that $u_0$ is convex and defined in $B_3(0)$. Define for $\epsilon > 0$, $b(x,t) := u_0(x) - \epsilon\left(\vert x \vert - \frac{1}{2}\right)_+^4+\sqrt{1 + \vert Du_0 \vert^2} H[u_0]^\rho \cdot t$, where $H[u_0]$ denotes the mean curvature of $\graph u_0$. Fix $\epsilon > 0$ such that $b(\cdot,t)$ is strictly convex in $B_2(0)$. Consider the boundary value problem
\begin{align*}
\begin{cases}
  \dot{u} = \sqrt{1 + \vert Du \vert^2} H^\rho,  & \text{in } B_1(0) \times [0,T],\\
  u = b & \text{on } (\partial B_1(0) \times [0,T]) \cup (B_1(0) \times \{0\}).
\end{cases}
\end{align*}
This is a well-posed boundary value problem as $b$ is chosen so that compatibility conditions at the boundary are fulfilled. Hence there exists $\gamma > 0$, $T_b > 0$, we may assume that $T_b < T$, and a solution $u_b \in C^{2,\gamma;1,\gamma/2}\left(\overline{B_1(0)} \times [0,T_b]\right)$, depending only on $u_0$ in $B_2(x_0)$ and $\epsilon$. We may choose $T_b > 0$ smaller to ensure that there exists $\delta > 0$ such that $\dot{u} \geq \delta > 0$ and that $\graph u_b(\cdot,t)|_{\partial B_1(0)} \,\cap \graph u_0 = \emptyset$ for $t \in [0,T_b]$. Hence $u_b$ acts as a barrier form below.\par
Let $\phi(t)$ denote the point such that $Du(\phi(t),t) = 0$. Observe that $\phi(0) = 0$. Assume without loss of generality that $\phi$ is defined on $[0,T_b]$.  We have $u(\phi(T_b),t) - u(0,0) \geq \delta T_b$. For strictly convex solutions $u$, $\phi$ is differentiable. Moreover, $t \mapsto \psi(t) \equiv u(\phi(t),t)$ is monotone. We claim that
\begin{align*}
\dot{\psi}(t) \geq \frac{1}{2}\frac{\psi(T_b) - \psi(0)}{\rho + 1}T_b^{-\frac{1}{\rho + 1}} t^{-\frac{\rho}{\rho + 1}}
\end{align*}
for some $t \in [0,T_b]$.
Then the Harnack inequality \eqref{eq:harnack} implies that
\begin{align*}
\dot{\psi}(T_b) = \frac{\psi(T_b) - \psi(0)}{\rho + 1}\frac{1}{T_b}
\end{align*}
and the Lemma follows from \eqref{eq:harnack}, applied to $0 < T_b < T$. Otherwise, we get a contradiction as then
\begin{align*}
\psi(T_b) - \psi(0) =&\, \int_0^{T_b} \dot{\psi}(t) dt \leq \frac{\psi(T_b) - \psi(0)}{\rho + 1} T_b^{-\frac{1}{\rho + 1}} \int_0^{T_b} t^{-\frac{\rho}{\rho + 1}} dt\\ =&\, \frac{1}{2} \left(\psi(T_b) - \psi(0)\right).
\end{align*}
The bounds can be chosen in a continuous way as a similar argument gives a bound near $x_0$. 
\end{proof}

\section{Longtime Existence}\label{sc:longtime}

\begin{proof}[Proof of Theorem \ref{th:existence}:]
We may assume that $u(x) \rightarrow \infty$ as $\vert x \vert \rightarrow \infty$. We approximate $\graph u_0$ by a sequence of closed strictly convex $C^{2,\beta}$-hypersurfaces $M^k$ above $\graph u_0$ such that $M^k$ and $\graph u_0$ coincide in $\{x^{n+1} < k\}$. 
According to Schulze \cite{fs:evolution}, for each $k$, there exists a smooth strictly convex solution $M_t^k$ to \eqref{eq:mean} with $M_0^k = M^k$ which contracts to a point in finite time. The approximations $M^k$ can be chosen such that there exist functions $u^k$ fulfilling $M_t^k = \graph u^k (\cdot,t)$ in $\{x^{n+1} < k\}$ for $0 \leq t < k$. Using large spheres as barriers, it is easy to see that the constructions in Section \ref{sc:coordinate}, especially Lemma \ref{lm:comparisonII}, work for the approximating hypersurfaces $M_t^k$ if $k$ is sufficiently large. In particular, they imply locally uniform $C^1$-estimates. Fix $R,T > 0$. According to Lemma \ref{lm:coordinate}, we can choose coordinate systems that allow for applying the local $C^2$-estimates of Section \ref{sc:local}. We obtain $C^2$-estimates on $\partial B_r(0)$ for some $r \geq R$: $$0\cdot \delta_{ij} \leq u_{ij}^k \leq c \cdot \delta_{ij} \text{ on } \partial B_r(0) \times [0,T].$$ \par Inside $B_r(0)$, we use \cite{fs:evolution} again to get similar estimates. As lower velocity bounds follow from Section \ref{sc:lower}, we see that \eqref{eq:mean} is a strictly parabolic equation.  Applying Schauder theory  and estimates by Krylov-Safonov \cite{os:pde} we achieve higher regularity. Using the Arcel\`a-Ascoli theorem, we find a subsequence $u^{k_l}$ which converges to a solution of \eqref{eq:mean}. For $t > 0$, the subsequence converges smoothly.
\end{proof}

\bibliographystyle{amsplain} 
\def\weg#1{} \def\unterstrich{\underline{\rule{1ex}{0ex}}} \def\cprime{$'$}
  \def\cprime{$'$} \def\cprime{$'$} \def\cprime{$'$}
\providecommand{\bysame}{\leavevmode\hbox to3em{\hrulefill}\thinspace}
\providecommand{\MR}{\relax\ifhmode\unskip\space\fi MR }
\providecommand{\MRhref}[2]{%
  \href{http://www.ams.org/mathscinet-getitem?mr=#1}{#2}
}
\providecommand{\href}[2]{#2}

\end{document}